\documentclass[12pt,a4paper]{article}
\usepackage[left=1in,right=1in,top=1in,bottom=1in,a4paper]{geometry}
\usepackage{amsfonts,amsmath,amsthm}
\usepackage{graphicx,dsfont}
\usepackage{pst-plot,pstricks,multido}
\usepackage{enumerate,enumitem}
\usepackage{amssymb}
\usepackage{hyperref}

\newtheorem{thm}{Theorem}[section]
\newtheorem{lem}[thm]{Lemma}
\newtheorem{cor}[thm]{Corollary}

\theoremstyle{definition}
\newtheorem{defn}[thm]{Definition}

\newtheorem{ex}[thm]{Example}
\newtheorem{remk}[thm]{Remark}

\newcommand{\kc}{\mathcal{C}}
\newcommand{\pc}{\mathcal{P}_\mathcal{C}}

\makeatletter
\newcommand{\specificthanks}[1]{\@fnsymbol{#1}}
\makeatother

\title{A link between Kendall's $\tau$, the length measure and the surface of bivariate copulas, 
and a consequence to copulas with self-similar support}

\author{Juan Fern\'andez-S\'anchez\thanks{Grupo de Investigaci\'on de An\'alisis Matem\'atico, Universidad de Almer\'{\i}a, La Ca\~nada de San Urbano, Almer\'{\i}a, Spain, E-mail: juanfernandez@ual.es} \and
Wolfgang Trutschnig\thanks{Department for Artificial Intelligence and Human Interfaces, University of Salzburg, Austria,
E-Mail: wolfgang@trutschnig.net (corresponding author)}}

\begin{document}

\maketitle
\begin{abstract}
Working with shuffles we establish a close link between 
Kendall's $\tau$, the so-called length measure, and the surface area of bivariate copulas and derive some consequences.  
While it is well-known that Spearman's $\rho$ of a bivariate copula $A$ is a rescaled version of the volume of the area under 
the graph of $A$, in this contribution we show that the other famous concordance measure, Kendall's $\tau$, allows for a simple geometric interpretation as well - it is inextricably linked to the surface area of $A$.  
\end{abstract}

\section{Introduction}
Spearman's $\rho$ of a bivariate copula $A$ is a rescaled version of the volume below the graph of $A$ 
(see \cite{DS,Nel}) in the sense that
$$
\rho(A)= 12 \int_{[0,1]^2} A \,d\lambda_2 - 3
$$
holds. Letting $[A]_t :=\{(x,y) \in [0,1]^2: A(x,y) \geq t\}$ denote the lower $t$-cut of $A$ for every $t \in [0,1]$ 
and applying Fubini's theorem directly yields
$$
\rho(A)= 12 \int_{[0,1]} \lambda_2([A]_t)d\lambda(t) - 3,
$$
which lead the authors of \cite{CGST} to conjecturing that adequately rescaling the so-called length 
measure $\ell(A)$ of $A$, defined
as the average arc-length of the contour lines of $A$, might result in a (new or already known) concordance measure. 
The conjecture was falsified in \cite{CGST}, only some but not all properties of a concordance measure 
are fulfilled, in particular, we do not have continuity with respect to pointwise convergence of copulas in general.  

Motivated by the afore-mentioned facts, the objective of this note is two-fold: we first derive the somewhat 
surprising result that on a subfamily of 
bivariate copulas - the class $\kc_{mcd}$ of all mutually completely dependent copulas (including all 
classical shuffles) - which is dense  
in the class $\kc$ of all bivariate copulas with respect to uniform convergence, the length measure is, in fact, an 
affine transformation of Kendall's $\tau$ and vice versa. As a consequence, the length measure 
restricted to $\kc_{mcd}$ is continuous with respect to pointwise convergence of copulas. 
We then focus on the surface area of bivariate copulas and derive analogous statements, i.e., 
that on the class $\kc_{mcd}$ the surface area is an affine transformation of Kendall's $\tau$ 
(and hence of the length measure) too. For obtaining both main results a simple geometric identity 
linking the length measure and the surface area with the area of the set $\Omega_{\sqrt{2}}$, given by
\begin{equation}
\Omega^{A_h}_{\sqrt{2}} = \left\{(x,y) \in [0,1]^2: h(x) \leq y, \, h^{-1}(y) \leq x \right\},
\end{equation}
where $h$ denotes the transformation corresponding to the completely dependent copula $A_h$, will be key. 
An application to calculating Kendall's $\tau$, the length measure and 
the surface area of completely dependent copulas with self-similar support concludes the paper.

\section{Notation and preliminaries}
In the sequel we will let $\mathcal{C}$ denote the family of all bivariate copulas. For each copula 
$C\in\mathcal{C}$ the corresponding doubly stochastic measure will be
denoted by $\mu_C$, i.e., $\mu_C([0,x] \times [0,y]) = C(x,y)$ holds for all $x,y \in [0,1]$.
Considering the uniform metric $d_\infty$ on $\mathcal{C}$ it is well-known that $(\mathcal{C}, d_\infty)$ is a 
compact metric space and that in $\mathcal{C}^d$ pointwise and uniform convergence are equivalent. 
For more background on copulas and doubly stochastic measures we refer to \cite{DS, Nel}.\\
For every metric space $(\Omega,d)$ the Borel $\sigma$-field in $\Omega$ will be denoted by $\mathcal{B}(\Omega)$. 
The Lebesgue measure on the Borel $\sigma$-field $\mathcal{B}([0,1]^2)$ of $[0,1]^2$ will be denoted by 
$\lambda_2$, the univariate version on $\mathcal{B}([0,1])$ by $\lambda$. 
Given probability spaces $(\Omega,\mathcal{A},\mathbb{P})$ and $(\Omega',\mathcal{A}',\mathbb{P})'$ and a 
measurable transformation $T: \Omega \rightarrow \Omega'$ the push-forward of 
$\mathbb{P}$ via $T$ will be denoted by $\mathbb{P}^T$, i.e., $\mathbb{P}^T(F)=\mathbb{P}(T^{-1}(F))$ for 
all $F \in \mathcal{A}'$. \\
In what follows, Markov kernels will be a handy tool. A mapping 
$K\colon\mathbb{R} \times \mathcal{B}(\mathbb{R}) \to [0,1]$ is called a Markov kernel from $(\mathbb{R},\mathcal{B}(\mathbb{R}))$ to $(\mathbb{R},\mathcal{B}(\mathbb{R}))$ if the mapping $x \mapsto K(x,B)$ is measurable for every fixed $B \in \mathcal{B}(\mathbb{R})$ and the mapping 
$B \mapsto K(x, B)$ is a probability measure for every fixed $x \in \mathbb{R}$. A Markov kernel $K\colon\mathbb{R} \times \mathcal{B}(\mathbb{R}) \to [0,1]$ is called regular conditional distribution of a (real-valued) random variable $Y$ given (another random variable) $X$ if for every $B \in \mathcal{B}(\mathbb{R})$ 
\begin{align*}
K(X(\omega),B) = \mathbb{E}(\mathds{1}_B \circ Y | X)(\omega)
\end{align*}
holds $\mathbb{P}$-a.s. It is well known that a regular conditional distribution of $Y$ given $X$ exists and is unique 
$\mathbb{P}^X$-almost surely. For every $A \in \mathcal{C}$ (a version of) the corresponding regular conditional distribution (i.e., the regular conditional distribution of $Y$ given $X$ in the case that $(X,Y)\sim A$) will be denoted by 
$K_A(\cdot, \cdot)$ and directly be interpreted as mapping from $K_A: \,[0,1] \times \mathcal{B}([0,1]) \rightarrow [0,1]$. 
Note that for every $A \in \mathcal{C}$ and Borel sets $E,F \in \mathcal{B}([0,1])$ we have the following disintegration formulas:
\begin{align}\label{eq:cop_margin}
\int_E K_A(x,F) d\lambda(x) = \mu_A(E \times F) \quad \text{ and } \quad \int_{[0,1]} K_A(x, F) d\lambda(x) = \lambda(F)
\end{align}
For more details and properties of conditional expectations and regular conditional distributions we refer to  \cite{KALLENBERG,Klenke}. \\
A copula $A \in \kc$ will be called completely dependent if there exists some $\lambda$-preserving
transformation $h: [0,1] \rightarrow [0,1]$ (i.e., a transformation with $\lambda^h=\lambda$) such that 
$K(x,E)=\mathbf{1}_E(h(x))$ is a Markov kernel of $A$. The copula induced by $h$ will be denoted by $A_h$, 
the class of all completely dependent copulas by $\kc_{cd}$. 
A completely dependent copula $A_h$ is called mutually completely dependent, if
the transformation $h$ is bijective. Notice that in this case the transpose $A_h^t$ of $A_h$, defined by
$A_h^t(x,y)=A_h(y,x)$, coincides with $A_{h^{-1}}$. The family of all mutually completely dependent copulas will 
be denoted by $\kc_{mcd}$. It is well known (see \cite{DS,Nel}) that $\kc_{mcd}$ is dense in $(\kc,d_\infty)$, in fact
even the family of all equidistant even shuffles (again see \cite{DS,Nel}) is dense. For 
further properties of completely dependent copulas we refer to \cite{T06} and the references therein.  

Turning towards the length profile introduced and studied in \cite{CGST}, let $\Gamma_{A,t}$ denote the boundary 
of the lower $t$-cut $[A]_t$ in $(0,1)^2$ and $H_1(\Gamma_{A,t})$ it's arc-length. 
Then the length profile of $A$ is defined as the function
$L_A: [0,1] \rightarrow [0,\infty)$, given by 
\begin{equation}
L_A(t)=H_1(\Gamma_{A,t}).
\end{equation}
It is straightforward to show see that
$$
\sqrt{2}(1-t) \leq L_A(t) \leq 2(1-t)
$$
holds for every $t \in (0,1)$. Building upon $L_A$ the so-called length measure $\ell(A)$ of $A$ is defined as
\begin{equation}
\ell(A)= \int_{(0,1)} L_A(t) d\lambda(t)
\end{equation}
and describes the average arc-length of upper $t$-cuts of $A$. It is straightforward to verify that
\begin{equation}\label{eq:ineq}
\ell(W)=\frac{1}{\sqrt{2}} \leq \ell(A) \leq 1 = \ell(M)
\end{equation} 
as well as $\ell(A) \in [\frac{1}{\sqrt{2}},1]$ holds (ineq. (\ref{eq:ineq}) was also one of the reasons for falsely 
conjecturing that the length measure might be transformable into a concordance measure). 

In \cite{CGST} it was shown that for mutually completely dependent copulas $A_h$ the length profile allows for a simple 
calculation. In fact, using the co-area formula we have 
$$
\ell(A_h)= \int_{(0,1)^2} \Vert \nabla A_h(u,v) \Vert_2 \, d\lambda_2(u,v),
$$
where $\nabla A_h$ denotes the gradient of $A_h$ (whose existence $\lambda_2$-almost everywhere 
is assured by Rademacher's theorem and Lipschitz continuity, see \cite{EG}).
The last equation simplifies to the nice identity
\begin{equation}\label{eq:lm}
\ell(A_h) = 1-(2-\sqrt{2})\, \lambda_2(\Omega_{\sqrt{2}}),
\end{equation} 
with 
\begin{eqnarray}
\Omega_{\sqrt{2}}^{A_h}:=\Omega_{\sqrt{2}} &=& 
 \left\{(u,v) \in (0,1)^2: \Vert \nabla A_h(u,v) \Vert_2 =\sqrt{2}\right\} \nonumber  \\ 
&=& \left\{(u,v) \in (0,1)^2: h(u) \leq v, \, h^{-1}(v) \leq u \right\}
\end{eqnarray}
Throughout the rest of this note we will only write $\Omega_{\sqrt{2}}$ instead of 
$\Omega_{\sqrt{2}}^{A_h}$ whenever no confusion will arise.
Notice that for classical equidistant straight shuffles eq. (\ref{eq:lm}) implies that $\ell(A_h)$ can be calculated by simply
counting squares as Figure \ref{fig:Omega} illustrates in terms of two simple examples - one shuffle with 
three, and a second one with nice equidistant stripes.
\begin{figure}[!h]
      \begin{center}
     \includegraphics[width = \textwidth]{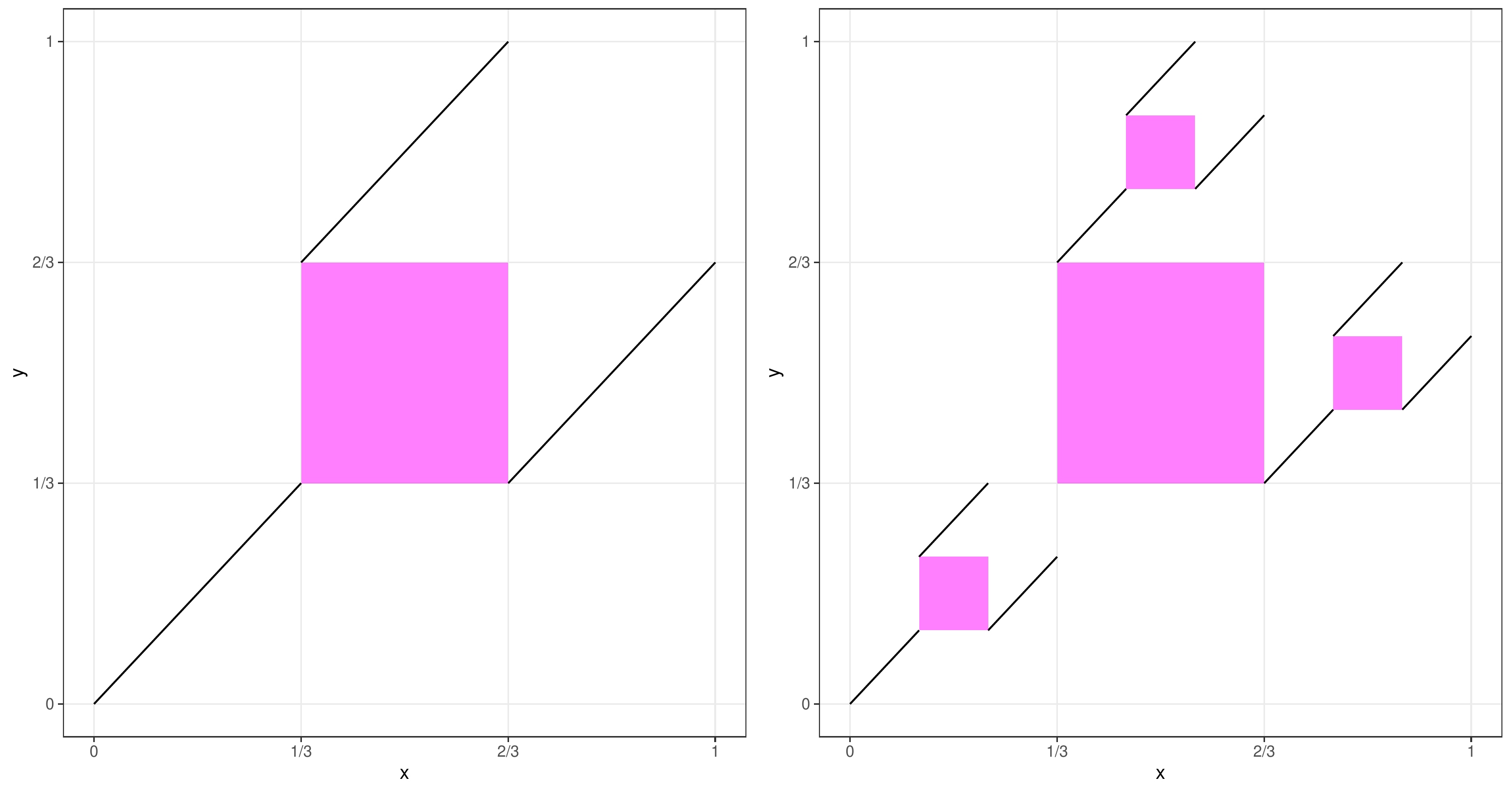}
    \caption{The set $\Omega_{\sqrt{2}}$ (in magenta) for an even shuffles of three strips 
    (left panel) and nine strips (right panel).
     In this case we have $\lambda_2(\Omega_{\sqrt{2}})=\frac{1}{9}$ and $\ell(A_h)=1-(2-\sqrt{2})\frac{1}{9}$
     for the first shuffle and $\lambda_2(\Omega_{\sqrt{2}})=\frac{1}{9} + \frac{1}{27}$ as well as $\ell(A_h)=1-(2-\sqrt{2})\left(\frac{1}{9} + \frac{1}{27} \right)$ for the second one.}
    \label{fig:Omega}
\end{center}
\end{figure}

\section{The interrelations}\label{sec:inter}
We now derive a simple formula linking Kendall's $\tau$ and the length measure for mutually completely dependent copula
and start with some preliminary observations. Working with checkerboard copulas, using integration by 
parts (see \cite{Nel}) and finally applying an approximation result like
\cite[Theorem 3.2]{KFT} yields that for arbitrary bivariate copulas $A,B \in \kc$
the following identity holds:
\begin{equation}\label{eq:tau.kern}
\tau(A)=4\int_{[0,1]^2} A d\mu_A -1 = 4 \left(\frac{1}{2} - \int_{[0,1]^2} K_A(x,[0,y]) K_{A^t}(y,[0,x]) d\lambda_2(x,y)\right)-1.
\end{equation}
For $A_h \in \kc_{mcd}$ eq. (\ref{eq:tau.kern}) can be derived in the following simple alternative way, 
which we include for the sake of completeness: 
Using the fact that for $A_h \in \kc_{mcd}$ and every $x \in [0,1]$ we have
\begin{eqnarray*}
A(x,h(x)) &=& \int_{[0,x]} K_{A_h}(t,[0,h(x)]) d\lambda(t) = \int_{[0,x]} K_{A_h}(t,[0,h(x))) d\lambda(t)\\
&=& \int_{[0,x]} \mathbf{1}_{[0,h(x))}(h(t)) d\lambda(t) = \int_{[0,x]} (1-\mathbf{1}_{[h(x),1]}(h(t)) d\lambda(t) \\
&=& x - \int_{[0,1]} \mathbf{1}_{[0,x]}(t) \mathbf{1}_{[h(x),1]}(h(t)) d\lambda(t).
\end{eqnarray*}
Using disintegration and change of coordinates directly yields
\begin{eqnarray*}
\int_{[0,1]^2} A d\mu_A &=& \int_{[0,1]} A(x,h(x)) d\lambda(x) = 
\frac{1}{2}-\int_{[0,1]} \int_{[0,1]} \mathbf{1}_{[0,x]}(t) \mathbf{1}_{[h(x),1]}(h(t)) d\lambda(t) d\lambda(x) 
\end{eqnarray*}
and hence proves eq. (\ref{eq:tau.kern}). 
The latter identity, however, boils down to an affine transformation of $\lambda_2(\Omega_{\sqrt{2}})$ by considering
\begin{eqnarray*}
\int_{[0,1]^2} A d\mu_A &=& \frac{1}{2} - \int_{[0,1]} \int_{[0,1]} \mathbf{1}_{[0,x]}(h^{-1} \circ h(t)) \mathbf{1}_{[h(x),1]}(h(t)) d\lambda(t) d\lambda(x) \\
&=& \frac{1}{2} -  \int_{[0,1]} \int_{[0,1]}  \mathbf{1}_{[0,x]}(h^{-1}(y)) \mathbf{1}_{[h(x),1]}(y) d\lambda(y) d\lambda(x) \\
&=& \frac{1}{2} - \lambda_2(\Omega_{\sqrt{2}}). 
\end{eqnarray*}
Having this, the identity
\begin{equation}\label{eq:tauAh}
\tau(A_h) = 4 \left(\frac{1}{2} - \lambda_2(\Omega_{\sqrt{2}}) \right) - 1 = 1- 4 \lambda_2(\Omega_{\sqrt{2}}).
\end{equation}
follows immediately. Notice that eq. (\ref{eq:tauAh}) implies that the area of $\Omega_{\sqrt{2}}$ coincides with the 
quantity $\textrm{inv(h)}$ as studied in \cite[Lemma 3.1]{SPT}. 
Comparing eq. (\ref{eq:lm}) and eq. (\ref{eq:tauAh}) shows the existence of an affine transformation 
$a:[-1,1]\rightarrow [\frac{1}{\sqrt{2}},1]$ such that 
$$
a\left(\tau(A_h)\right) = \ell(A_h)
$$
holds for every $A_h \in \mathcal{C}_{mcd}$ - in other words, we have proved the subsequent result:
\begin{thm}\label{thm:main1}
For every $A_h \in \kc_{mcd}$ the following identity linking the length measure $\ell$ and Kendall's $\tau$ holds:
\begin{equation}
\ell(A_h)= 1 - \frac{2-\sqrt{2}}{4} (1-\tau(A_h))
\end{equation}
\end{thm}
\noindent Theorem \ref{thm:main1} provides an answer to the question posed in \cite{CGST}, `whether there are links between the 
length of level curves and concordance measures' - even the conjectured `weighting' mentioned in \cite{CGST}
is not necessary, in the class $\kc_{mcd}$ all we need is a fixed affine transformation. 

\noindent In \cite{CGST} it was further shown that the length measure interpreted as function $\ell: \mathcal{C} \rightarrow [\frac{\sqrt{2}}{2},1]$ 
is not continuous w.r.t. $d_\infty$. The previous result implies, however, that within the dense subclass $\kc_{mcd}$
the length measure is indeed continuous:
\begin{cor}
The mapping $\ell: \mathcal{C}_{mcd} \rightarrow [\frac{1}{\sqrt{2}},1]$ is continuous with respect to  $d_\infty$. 
\end{cor}
\begin{proof}
Suppose that $A_h,A_{h_1},A_{h_2},\ldots$ are mutually completely dependent copulas and that 
the sequence $(A_{h_n})_{n \in \mathbb{N}}$ converges to $A_h$ pointwise. Being a concordance measure Kendall's $\tau$
is continuous with respect to  $d_\infty$, so we have $\lim_{n \rightarrow \infty} \tau(A_{h_n})=\tau(A_h)$ and eq. (\ref{eq:tauAh})
directly yields $\lim_{n \rightarrow \infty} \ell(A_{h_n})=\ell(A_h)$. 
\end{proof}

\begin{cor}
For every $z \in [\frac{1}{\sqrt{2}},1]$ there exists some mutually completely dependent 
copula $A_h$ with $\ell(A_h)=z$. In other words,
all values in $[\frac{1}{\sqrt{2}},1]$ are attained by $\ell$.  
\end{cor}
\begin{proof}
According to \cite{SPT} for each $(x,y)$ in the region determined by Kendall' $\tau$ and Spearman's $\rho$ there exists 
some mutually completely dependent copula $C_h$ fulfilling 
$$
(\tau(A_h),\rho(A_h))=(x,y).
$$ 
Having this, the result directly follows via eq. (\ref{eq:tauAh}).
\end{proof}

Moving away from the length measure we now turn to the surface area of copulas, derive analogous statements and 
start with showing yet another simple formula for elements in $\kc_{mcd}$. 
Considering that copulas are Lipschitz continuous, the 
surface area $\textrm{surf}(A)$ of an arbitrary copula $A$ is given by 
\begin{eqnarray}\label{eq:surface.gen}
\textrm{surf}(A) &=& \int_{[0,1]^2} \sqrt{\left(\frac{\partial A}{ \partial x}(x,y)\right)^2 + 
\left(\frac{\partial A}{ \partial y}(x,y)\right)^2 + 1} \, \, d\lambda_2(x,y) \nonumber \\
&=& \int_{[0,1]^2} \sqrt{K_A(x,[0,y])^2 + K_{A^t}(y,[0,x])^2 + 1} \, \, d\lambda_2(x,y). 
\end{eqnarray}
Again working with mutually completely dependent copulas yields the following result:
\begin{lem}
For every $A_h \in \kc_{mcd}$ the surface area of $A_h$ is given by
\begin{equation}\label{eq:surface}
\textrm{surf}(A_h) = \sqrt{2} - \left(2 \sqrt{2} - 1 - \sqrt{3}\right) \lambda_2(\Omega_{\sqrt{2}}).
\end{equation}
\end{lem}
\begin{proof}
For the case of a completely dependent copula $A_h$ eq. (\ref{eq:surface.gen}) obviously simplifies to 
\begin{eqnarray*}
\textrm{surf}(A_h) &=& \int_{[0,1]^2} \sqrt{\mathbf{1}^2_{[0,y]}(h(x))  + \mathbf{1}^2_{[0,x]}(h^{-1}(y)) + 1} \, \, d\lambda_2(x,y) \\
&=& \int_{[0,1]^2} \sqrt{\mathbf{1}_{[0,y]}(h(x))  + \mathbf{1}_{[0,x]}(h^{-1}(y)) + 1} \, \, d\lambda_2(x,y).
\end{eqnarray*}
Considering that the latter integrand is a step function only attaining the values $1,\sqrt{2}$ and $\sqrt{3}$, 
defining  
\begin{eqnarray*}
\Omega_{0}^{A_h}:=\Omega_{0} &=& \left\{(x,y) \in (0,1)^2: h(x) > y, \, h^{-1}(y)>x \right\}
\end{eqnarray*}
as well as ($\Omega_{\sqrt{2}}$ as before)
\begin{eqnarray*}
\Omega_{1}^{A_h}:=\Omega_{1} &=& [0,1]^2 \setminus (\Omega_{\sqrt{2}} \cup \Omega_{0}) 
\end{eqnarray*}
we therefore have
\begin{eqnarray*}
\textrm{surf}(A_h) &=& \lambda_2(\Omega_0) + \sqrt{2} \, \lambda_2(\Omega_1) + \sqrt{3} \,\lambda_2(\Omega_{\sqrt{2}}).
\end{eqnarray*}
The latter identity can be further simplified: The measurable bijection $\Psi_h: [0,1]^2 \rightarrow [0,1]^2$, 
defined by $\Psi_h(x,y)=(h^{-1}(y),h(x))$ obviously fulfills $\lambda_2^{\Psi_h}=\lambda_2$. Therefore using the fact that
\begin{eqnarray*}
\Psi_h^{-1}(\Omega_0) &=& \{(x,y) \in [0,1]^2: (h^{-1}(y),h(x)) \in \Omega_0\} \\
 &=& \{(x,y) \in [0,1]^2: h(x) \leq y, h^{-1}(y) \leq x\} = \Omega_{\sqrt{2}} 
\end{eqnarray*}
it follows that $\lambda_2(\Omega_0)=\lambda_2(\Omega_{\sqrt{2}})$ holds. This altogether yields 
\begin{eqnarray*}
\textrm{surf}(A_h) &=& \lambda_2(\Omega_{\sqrt{2}}) + \sqrt{2} \, \lambda_2(\Omega_1) + \sqrt{3} \,\lambda_2(\Omega_{\sqrt{2}}) \\
&=& (1+\sqrt{3}) \, \lambda_2(\Omega_{\sqrt{2}}) + \sqrt{2}\left(1- 2 \lambda_2(\Omega_{\sqrt{2}}) \right) \\
&=& \sqrt{2} + \underbrace{\left( 1+\sqrt{3} - 2 \sqrt{2} \right)}_{<0}\, \lambda_2(\Omega_{\sqrt{2}}),
\end{eqnarray*}
which completes the proof.
\end{proof}
\begin{thm}
For every $A_h \in \kc_{mcd}$ the following identity linking the surface area and Kendall's $\tau$ holds:
\begin{equation}\label{eq:surf.Ah}
\textrm{surf}(A_h)= \sqrt{2} - \frac{2\sqrt{2} - 1 - \sqrt{3}}{4} (1-\tau(A_h))
\end{equation}
\end{thm}

\noindent As in the case of the length measure we have the following two immediate corollaries:
\begin{cor}
The mapping $\textrm{surf}: \mathcal{C}_{mcd} \rightarrow [\frac{1+\sqrt{3}}{2},\sqrt{2}]$ is continuous with respect to  $d_\infty$. 
\end{cor}
\begin{cor}
For every $z \in [\frac{1+\sqrt{3}}{2},\sqrt{2}]$ there exists some mutually completely copula 
$A_h$ with $\textrm{surf}(A_h)=z$. In other words, all values in $[\frac{1+\sqrt{3}}{2},\sqrt{2}]$ are attained by 
$\textrm{surf}$.  
\end{cor}
\begin{remk}
The afore-mentioned interrelations lead to the following seemingly new interpretation of the interplay 
between the two most well-known measures of concordance, Kendall's $\tau$ and Spearman's $\rho$, as studied in \cite{DS,DuSt,SPT} (and the references therein): Within the dense class 
$\mathcal{C}_{mcd}$ maximizing/minimizing Kendall's $\tau$ for a given value of Spearman's $\rho$
is equivalent to maximizing/minimizing the surface area of copulas for a given value of the volume.
Determining the exact $\tau$-$\rho$ region (for which according to \cite{SPT} considering all 
shuffles is sufficient) is therefore reminiscent of the famous isoperimetric inequality bounding the surface area
of a set by a function of the volume (see \cite{Fe}).
\end{remk}

\section{Calculating $\tau, \ell$ and $\textit{surf}$ for mutually completely dependent copulas with 
self-similar support}
We first recall the notion of so-called transformation matrices and the construction of copulas with fractal/self-similar support,
then use these tools to construct mutually completely dependent copulas with self-similar support
and finally derive simple expressions for Kendall's $\tau$ and the length measure of copulas of this type. 

\begin{defn}[\cite{FNRL,T06,FST12}] \label{trafom}
An $n\times m$- matrix $T=(t_{ij})_{i=1\ldots n, \,j=1\ldots m }$ is called 
\emph{transformation matrix} if it fulfills the following four conditions: 
 (i) $\max(n,m)\geq 2$, (ii), all entries are non-negative, (iii) $\sum_{i,j} t_{ij}=1$, and (iv) no row or column
 has all entries $0$.
\end{defn}
\noindent In other words, a transformation matrix is a probability distribution $\tau$ on $(\mathcal{I},2^\mathcal{I})$ with
$\mathcal{I}=I_1 \times I_2$, $I_1=\{1,\ldots,n\}$ and $I_2=\{1,\ldots,m\}$, such that
$\tau(\{i\} \times I_2) >0$ for every $i \in I_1$ and $\tau(I_1 \times \{j\}) >0$ for every $j \in I_2$. 

\noindent Given a transformation matrix $T$ define the vectors
$(a_j)_{j=0}^m, (b_i)_{i=0}^n$ of cumulative column and row sums by
\begin{eqnarray}
a_0&=&b_0=0 \nonumber \\
a_j&=&\sum_{j_0\leq j} \sum_{i=1}^n t_{ij} \hspace{0.5cm} j\in \{1,\ldots,m\} \\
b_i&=&\sum_{i_0\leq i} \sum_{j=1}^m t_{ij} \hspace{0.5cm} i\in \{1,\ldots,n\}. \nonumber
\end{eqnarray}
Considering that $T$ is a transformation matrix both $(a_j)_{j=0}^m$ and $(b_i)_{i=0}^n$ are strictly increasing. Consequently 
$R_{ji}:=[a_{j-1},a_j]\times [b_{i-1}, b_i]$ are compact non-empty rectangles for every $j\in \{1,\ldots,m\}$ and
$i\in \{1,\ldots,n\}$. Defining the contraction 
 $w_{ji}:[0,1]^2 \rightarrow R_{ji}$ by
$$
  w_{ji}(x,y)=\big(a_{j-1} + x (a_j - a_{j-1})\,,\, b_{i-1} + x (b_i - b_{i-1})\big)
$$
therefore yields the IFSP $\{[0,1]^2,(w_{ji})_{j=1\ldots m,i=1\ldots n},(t_{ij})_{j=1\ldots m,i=1\ldots n} \}$.
The induced operator $V_T$ on $\pc$, given by   
\begin{equation} \label{Vop}
   V_T(\mu):= \sum_{j=1}^m \sum_{i=1}^n t_{ij}\,\mu^{w_{ji}}
  \end{equation}
is easily verified to be well-defined (i.e., it maps $\pc$ into itself, again see \cite{FST12,FNRL,T06}) - in the sequel we will 
therefore also consider $V_T$ as a transformation mapping $\kc$ into itself. 
According to \cite{T06} for every transformation matrix $T$ there exists a unique
copula $A_T^*$ with $V_T(A_T^*)=A_T^*$ such that  
\begin{equation}\label{eq:convD1}
\lim_{n \rightarrow \infty} D_1(V_T^n(B),A_T^*)=0
\end{equation}
holds for arbitrary $B \in \kc$ (i.e., $A_T^*$ is the unique, globally attractive fixed point of $V_T$).\\

Suppose now that $2\leq N \in \mathbb{N}$ and let $\pi$ be a permutation of $\{1,\ldots,N\}$. Then the matrix
$T_\pi=(t_{ij})_{i=1\ldots N, \,j=1\ldots N}$, defined by
$$
t_{i,j} = \frac{1}{N} \,\mathbf{1}_{\{j\}}(\pi(i)), \quad i,j \in \{1,\ldots,N\}
$$
is obviously a transformation matrix. To simplify notation we will simply write $V_\pi:=V_{T_\pi}$ as well as
$A_{T_\pi}^*=A_T^*$ in the sequel.
Obviously $V_\pi$ does not only map $\kc$ to $\kc$ but also $\kc_{mcd}$ to $\kc_{mcd}$. Considering that 
(see \cite{T06}) $\kc_{cd}$ is closed in $(\kc,D_1)$ using eq. (\ref{eq:convD1}) it follows immediately that
$A_\pi^* \in \kc_{mcd}$, so there exists some $\lambda$-preserving bijection $h_\pi^*$ with $A_\pi^*=A_{h^*}$.
Since the support of $A_\pi^*$ is self-similar it seems intractable to calculate $\ell(A_\pi^*), \textrm{surf}(A_\pi^*)$ 
and $\tau(A_\pi^*)$ for general $\pi$. The results established in the previous section, however, make it possible to derive 
simple expressions for both quantities.\\

We start with a simple illustrative example and then prove the general result (in a different manner).
\begin{ex}\label{ex1}
Consider $N=3$ and the permutation $\pi=(1,3,2)$. Since, firstly, $A_\pi^* \in \kc_{mcd}$, secondly, 
$A_\pi^*$ is globally attractive, and since, thirdly, $\lim_{n \rightarrow \infty} \tau(V_\pi^n(M))=\tau(A_\pi^*)$ 
implies $\lim_{n \rightarrow \infty} \lambda_2(\Omega_{\sqrt{2}}^{V_\pi^n(M)}) =\lambda_2(\Omega_{\sqrt{2}}^{A_\pi^*})$,  
it suffices to calculate $\lambda_2\left(\Omega_{\sqrt{2}}^{A_{\pi^*}}\right)$ which can be done as follows: 
Obviously we have (see Figure \ref{fig:Omega} for an illustration of the steps 3-6 in the construction)
\begin{eqnarray*}
\lambda_2\left(\Omega_{\sqrt{2}}^{V_\pi(M)}\right) &=& \frac{1}{9} \\
\lambda_2\left(\Omega_{\sqrt{2}}^{V^2_\pi(M)}\right) &=& \frac{1}{9}  + 3\, \frac{1}{9^2} = \frac{1}{9}\left(1+ \frac{1}{3}\right)\\
\lambda_2\left(\Omega_{\sqrt{2}}^{V^3_\pi(M)}\right) &=& \frac{1}{9}  + 3\, \frac{1}{9^2} + 9 \frac{1}{27^2} =
 \frac{1}{9}\left(1+ \frac{1}{3} + \frac{1}{3^2}\right) \\
 \vdots &=& \vdots \\
\lambda_2\left(\Omega_{\sqrt{2}}^{V^n_\pi(M)}\right) &=& \frac{1}{9}\left(1+ \frac{1}{3} + \ldots + \frac{1}{3^{n-1}}\right) 
\end{eqnarray*}
which yields
$$
\lambda_2\left(\Omega_{\sqrt{2}}^{A_{\pi^*}}\right) = \frac{1}{9} \frac{1}{1-\frac{1}{3}} = \frac{1}{6}.
$$
Having that, using eqs. (\ref{eq:lm}), (\ref{eq:tauAh}) and (\ref{eq:surf.Ah}) shows 
\begin{equation*}
\ell(A_\pi^*) = 1-(2-\sqrt{2}) \frac{1}{6}, \quad \tau(A_\pi^*) = 1- 4 \ \frac{1}{6} = \frac{1}{3}
\end{equation*}
as well as 
\begin{equation*}
\textrm{surf}(A_\pi^*)= \sqrt{2} - \left(2 \sqrt{2} - 1 - \sqrt{3}\right)\frac{1}{6}. 
\end{equation*}

\end{ex}

\begin{figure}[!h]
      \begin{center}
     \includegraphics[width = \textwidth]{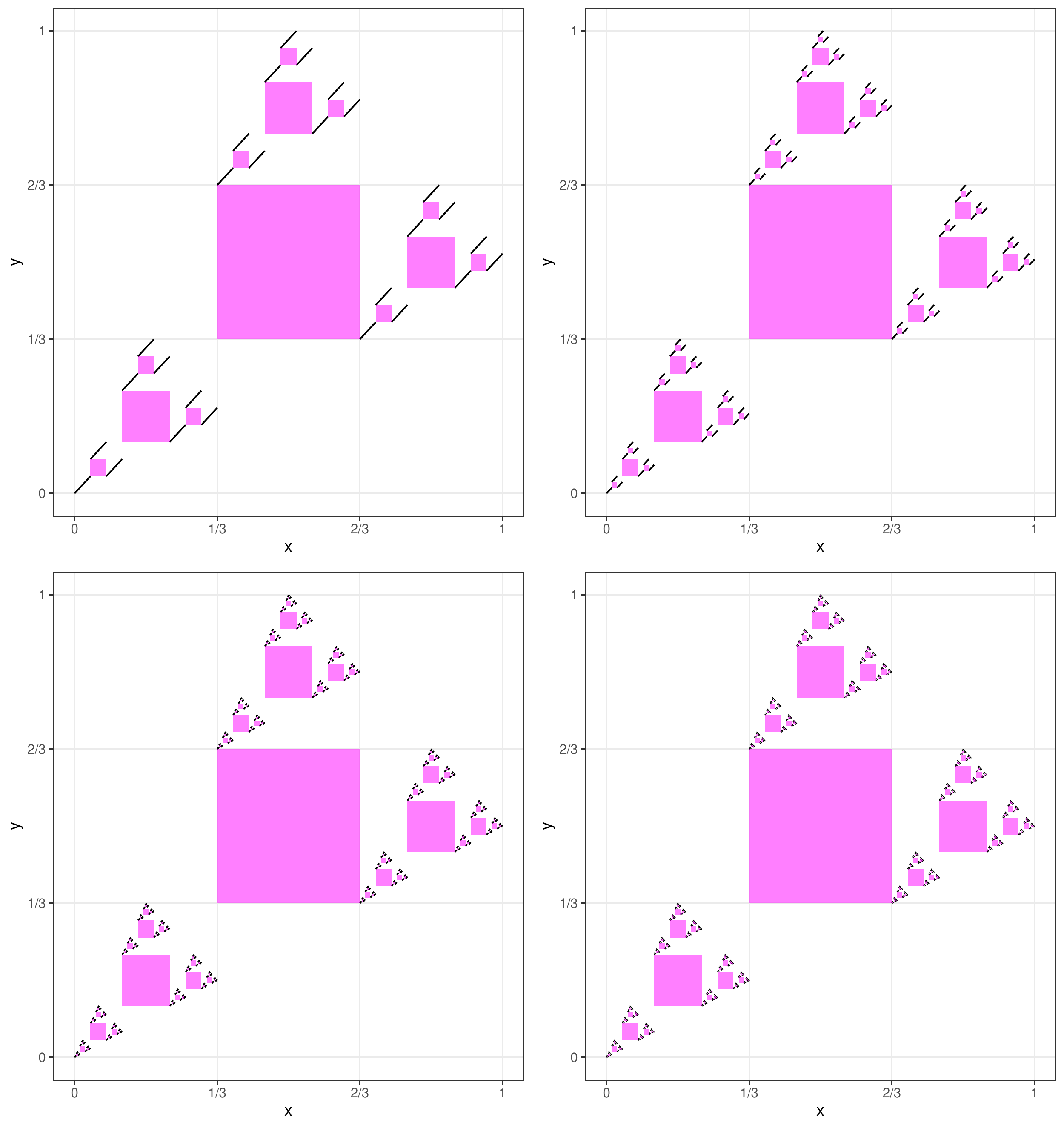}
    \caption{Supports of the copulas $V_\pi^n(M)$ (black line segments) 
    and the corresponding sets $\Omega_{\sqrt{2}}^{V_\pi^n(M)}$ (magenta squares)
    for $n \in \{3,4,5,6\}$ and $\pi=(1,3,2)$ as considered in Example \ref{ex1}.}
    \label{fig:Omega}
\end{center}
\end{figure}

\clearpage

\begin{thm}\label{thm:final}
Suppose that $N\geq 2$ and that $\pi$ is a permutation of $\{1,\ldots,N\}$. Then the following identities hold for
the copula $A_\pi^*$ with self-similar support:
\begin{eqnarray}\label{eq:tau_sss}
\tau\left( A_\pi^*\right) &=& 1- 4 \, \frac{N}{N-1} \lambda_2\left(\Omega_{\sqrt{2}}^{V_\pi(M)}\right)  \\
&=& 1- 4 \, \frac{1}{N(N-1)} \,\#\left\{(i,j) \in \{1,\ldots,N\}^2: \pi(i) <j \textrm{ and } \pi^{-1}(j)<i \right\} \nonumber \\
\ell\left( A_\pi^*\right) &=&  1- (2-\sqrt{2}) \, \frac{N}{N-1} \lambda_2\left(\Omega_{\sqrt{2}}^{V_\pi(M)}\right)  \\
&=& 1- (2-\sqrt{2})  \, \frac{1}{N(N-1)} \,\#\left\{(i,j) \in \{1,\ldots,N\}^2: \pi(i) <j \textrm{ and } \pi^{-1}(j)<i \right\} \nonumber \\
\textrm{surf}\left( A_\pi^*\right) &=&  \sqrt{2} - \left(2 \sqrt{2} - 1 - \sqrt{3}\right) \, \frac{N}{N-1} \lambda_2\left(\Omega_{\sqrt{2}}^{V_\pi(M)}\right) 
\end{eqnarray}
\end{thm}
\begin{proof}
First of all notice that for every $A_h \in \kc_{mcd}$ we have
\begin{equation}\label{eq:temp.final}
\lambda_2\left(\Omega_{\sqrt{2}}^{V_\pi(A_h)}\right) = \lambda_2\left(\Omega_{\sqrt{2}}^{V_\pi(M)}\right) + 
\frac{1}{N} \lambda_2\left(\Omega_{\sqrt{2}}^{A_h}\right).
\end{equation}
Since $A_\pi^*=A_{h^*}$ for some $\lambda$-preserving bijection $h^*$ and since $V_\pi(A_\pi^*)=A_\pi^*$ holds, 
eq. (\ref{eq:temp.final}) implies
$$
\lambda_2\left(\Omega_{\sqrt{2}}^{A_\pi^*}\right) = \lambda_2\left(\Omega_{\sqrt{2}}^{V_\pi(M)}\right) + 
\frac{1}{N} \lambda_2\left(\Omega_{\sqrt{2}}^{A_\pi^*}\right),
$$
from which we conclude
$$
\lambda_2\left(\Omega_{\sqrt{2}}^{A_\pi^*}\right)  = \frac{N}{N-1} \lambda_2\left(\Omega_{\sqrt{2}}^{V_\pi(M)}\right). 
$$
Having this, the desired identities follow by applying eqs. (\ref{eq:lm}), (\ref{eq:tauAh}), and (\ref{eq:surf.Ah}). 
\end{proof}

We conclude the paper with the following example.

\begin{ex}\label{ex2}
Consider $N=4$ and the permutation $\pi=(3,1,4,2)$. Figure \ref{fig:Omega2} depicts the first four steps
in the construction process of the corresponding copula $A_\pi^*$. 
Since in this case we have 
\begin{equation*}
\lambda_2\left(\Omega_{\sqrt{2}}^{V_\pi(M)}\right) = \frac{3}{16},
\end{equation*}
applying Theorem \ref{thm:final} directly yields
\begin{equation*}
\tau(A_\pi^*) = 0 ,\quad \ell(A_\pi^*) = \frac{1}{2} + \frac{\sqrt{2}}{4}
\end{equation*}
as well as 
\begin{equation*}
\textrm{surf}(A_\pi^*)= \frac{\sqrt{2}}{2} + \frac{1}{4} + \frac{\sqrt{3}}{4}. 
\end{equation*}

\end{ex}

\begin{figure}[!h]
      \begin{center}
     \includegraphics[width = \textwidth]{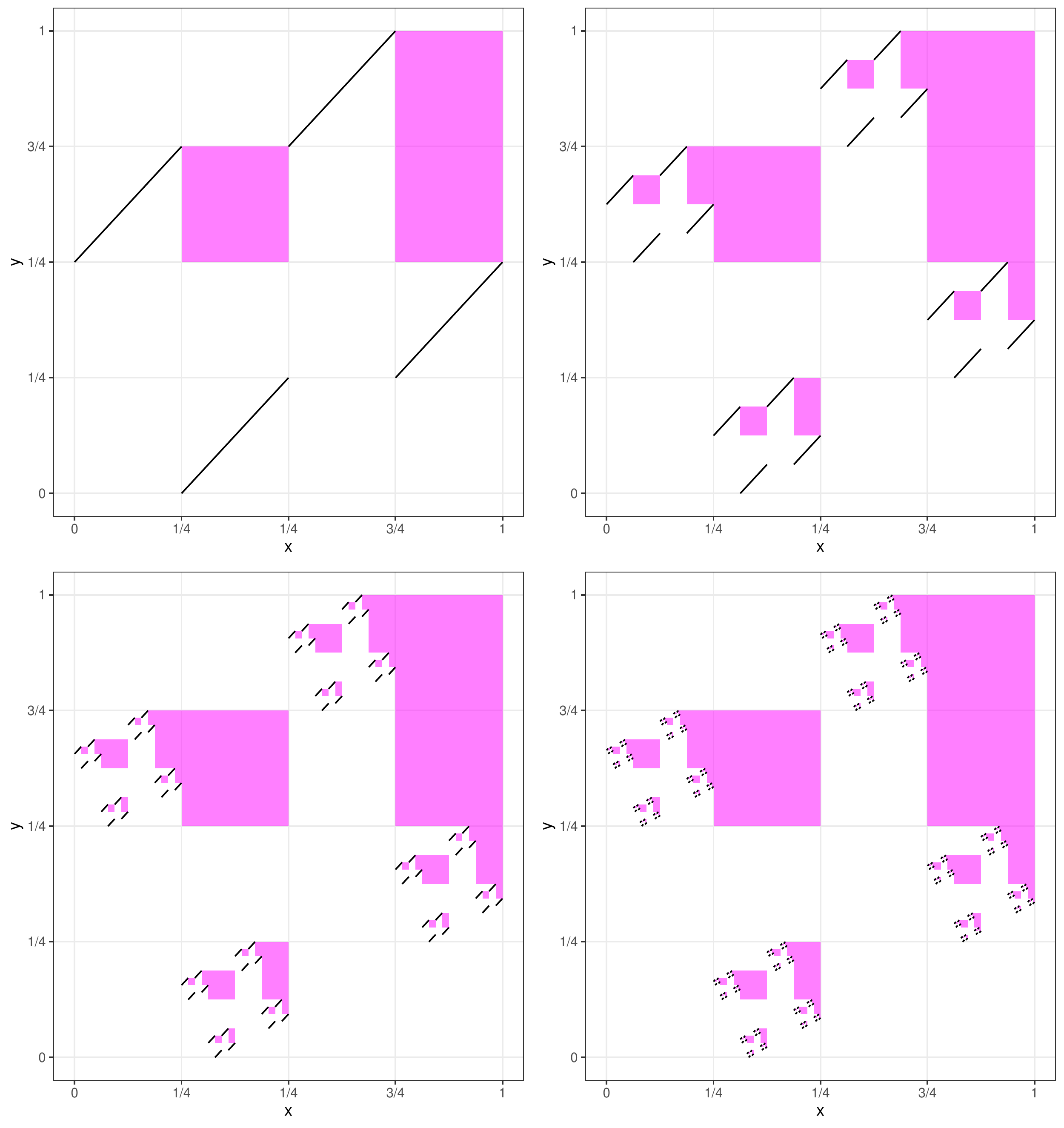}
    \caption{Supports of the copulas $V_\pi^n(M)$ (black line segments) 
    and the corresponding sets $\Omega_{\sqrt{2}}^{V_\pi^n(M)}$ (magenta rectangles)
    for $n \in \{1,2,3,4\}$ and $\pi=(3,1,4,2)$ as considered in Example \ref{ex2}.}
    \label{fig:Omega2}
\end{center}
\end{figure}

\vspace{1cm}

\noindent \textbf{Acknowledgement} \\
The second author gratefully acknowledge the support of the
WISS 2025 project `IDA-lab Salzburg’ (20204-WISS/225/197-2019 and 20102-F1901166-KZP).


\end{document}